\documentclass[12pt]{amsart}

\usepackage[T1]{fontenc}
\usepackage[utf8]{inputenc}
\usepackage{amssymb,amsmath,mathtools}
\usepackage{booktabs}
\usepackage{tikz}
\usepackage{comment}
\usepackage[hidelinks]{hyperref}

\makeatletter
\@namedef{subjclassname@2020}{%
	\textup{2020} Mathematics Subject Classification}
\makeatother

\newtheorem{theorem}{Theorem}
\newtheorem{lemma}{Lemma}
\newtheorem{proposition}{Proposition}
\newtheorem{corollary}{Corollary}

\newtheorem{remark}{Remark}

\begin{document}

\title[The exact $L_1$-norm supremum of Walsh--Kaczmarz--Fej\'er kernels]
{The exact $L_1$-norm supremum of Walsh--Kaczmarz--Fej\'er kernels}

\author[I. Blahota]{Istv\'an Blahota}

\address{Institute of Mathematics and Computer Sciences\\
	University of Ny\'\i regyh\'aza\\
	H-4400 Ny\'\i regyh\'aza, P.O. Box 166, Hungary}
\email{blahota.istvan@nye.hu}

\begin{abstract}
	We study the $L_1(G)$-norms of Fej\'er kernels for the
	Walsh--Kaczmarz system.  In the Walsh--Paley case the identity
	\[
	\|K^w_{2^n}\|_1=1 \qquad (n\in\mathbb N)
	\]
	follows directly from the binary structure.  Toledo's sharp result gives the
	global supremum
	\[
	\sup_{n\in\mathbb P}\|K^w_n\|_1=\frac{17}{15}.
	\]
	The Walsh--Kaczmarz kernels behave differently.  We first prove that
	\[
	\sup_{n\in\mathbb N}\|K^\kappa_{2^n}\|_1
	=
	\lim_{n\to\infty}\|K^\kappa_{2^n}\|_1
	=
	\frac43,
	\]
	and we prove that the sequence $(\|K^\kappa_{2^n}\|_1:n\in\mathbb N)$ is
	non-decreasing and is strictly increasing from $n=2$ on.  Then we use this
	result, an exact splitting formula from Skvortsov's decomposition, and the
	sharp Walsh--Paley estimates.  This gives the exact global supremum
	\[
	\sup_{n\in\mathbb P}\|K^\kappa_n\|_1=\frac{71}{50}.
	\]
\end{abstract}


\maketitle

\section{Notation and definitions}

Let $\mathbb P$ be the set of positive natural numbers and let
${\mathbb N}:={\mathbb P}\cup\{0\}$.  Let ${\mathbb Z}_2$ denote the cyclic
group of order $2$ with the discrete topology.  The group operation is addition
modulo $2$.  The dyadic group is
\[
	G:=\prod_{k=0}^{\infty}{\mathbb Z}_2.
\]
The group operation on $G$ is coordinate-wise addition, denoted by $+$.  The
elements of $G$ are sequences $x=(x_0,x_1,\dots)$, where
$x_k\in\{0,1\}$ for all $k\in\mathbb N$.

Let $\mu$ be the product measure of the measure $\mu_2$ on ${\mathbb Z}_2$ given
by $\mu_2(\{0\})=\mu_2(\{1\})=1/2$.  Then $\mu$ is the Haar measure on $G$.
Let $L_p(G)$ denote the usual Lebesgue spaces on $G$ with norm $\|\cdot\|_p$,
where $1\le p<\infty$.

Dyadic intervals are defined by
\[
	I_0(x):=G,
\]
\[
	I_n(x):=\{y\in G:y_0=x_0,\dots,y_{n-1}=x_{n-1}\},
\]
where $x\in G$ and $n\in\mathbb P$.  We write $I_n:=I_n(0)$ and
$J_n:=I_n\setminus I_{n+1}$.

The $n$th Rademacher function is defined by
\[
	r_n(x):=(-1)^{x_n},\qquad x\in G,\ n\in\mathbb N.
\]
Each $n\in\mathbb N$ has a unique dyadic expansion
\[
	n=\sum_{k=0}^{\infty}n_k2^k,
\]
where $n_k\in\{0,1\}$ and only finitely many $n_k$ are non-zero.  For
$n\in\mathbb P$ let
\[
	|n|:=\max\{k\in\mathbb N:n_k\ne0\},
\]
and put $|0|:=0$.

The Walsh--Paley functions are $w_0:=1$ and, for $n\in\mathbb P$,
\[
	w_n(x):=\prod_{k=0}^{\infty}r_k^{n_k}(x)
	=r_{|n|}(x)(-1)^{\sum_{k=0}^{|n|-1}n_kx_k}.
\]
The Walsh--Kaczmarz functions are $\kappa_0:=1$ and, for $n\in\mathbb P$,
\[
	\kappa_n(x):=r_{|n|}(x)\prod_{k=0}^{|n|-1}r_{|n|-1-k}^{n_k}(x)
	=r_{|n|}(x)(-1)^{\sum_{k=0}^{|n|-1}n_kx_{|n|-1-k}}.
\]
Then
\[
	r_n=w_{2^n}=\kappa_{2^n}\qquad(n\in\mathbb N).
\]

The Walsh--Kaczmarz and Walsh--Paley systems coincide on each dyadic block:
\begin{equation}\label{block}
	\{\kappa_n:2^k\le n<2^{k+1}\}
	=
	\{w_n:2^k\le n<2^{k+1}\}
\end{equation}
for all $k\in\mathbb N$, and $\kappa_0=w_0$.  Both systems are orthonormal and
complete in $L_2(G)$.

Skvortsov \cite{Sk} established a relation between the Walsh--Kaczmarz and
Walsh--Paley systems.  He used the transformation $\tau_n:G\longrightarrow G$ defined by
\[
	\tau_n(x):=(x_{n-1},x_{n-2},\dots,x_1,x_0,x_n,x_{n+1},\dots)
\]
for $n\in\mathbb N$, with the convention $\tau_0(x):=x$.  Notice that
$\tau_1(x)=x$ as well.  The relation is
\[
	\kappa_n(x)=r_{|n|}(x)w_{n-2^{|n|}}(\tau_{|n|}(x))
\]
for all $n\in\mathbb P$ and $x\in G$.  The transformation $\tau_n$ is
measure-preserving.

Let ${\mathcal P}^{\xi}_n$ be the collection of Walsh--Paley or
Walsh--Kaczmarz polynomials of order less than $n$.  Thus these are functions of the
form
\[
	P(x)=\sum_{k=0}^{n-1}c_k\xi_k(x),
\]
where $\xi_k=w_k$ or $\xi_k=\kappa_k$.  By \eqref{block},
\[
	{\mathcal P}_{2^n}:={\mathcal P}^w_{2^n}={\mathcal P}^{\kappa}_{2^n}.
\]
The collections of all Walsh--Paley and Walsh--Kaczmarz polynomials coincide.

For $\xi=w$ or $\xi=\kappa$, the $k$th Fourier coefficient, the $n$th partial
sum and the $n$th Fej\'er mean are
\[
	\widehat f^{\xi}(k):=\int_G f(x)\xi_k(x)\,d\mu(x),
\]
\[
	S_n^{\xi}(f;x):=\sum_{k=0}^{n-1}\widehat f^{\xi}(k)\xi_k(x),
	\qquad
	\sigma_n^{\xi}(f;x):=\frac1n\sum_{k=1}^{n}S_k^{\xi}(f;x),
\]
where $n\in\mathbb P$.

The $n$th Dirichlet and Fej\'er kernels with respect to the Walsh--Paley or
Walsh--Kaczmarz system are
\[
	D_n^{\xi}:=\sum_{k=0}^{n-1}\xi_k,
	\qquad
	K_n^{\xi}:=\frac1n\sum_{k=1}^{n}D_k^{\xi},
	\qquad n\in\mathbb P.
\]
Then
\[
	S_n^{\xi}(f;x)=\int_G f(u)D_n^{\xi}(x+u)\,d\mu(u),
\]
and
\[
	\sigma_n^{\xi}(f;x)=\int_G f(u)K_n^{\xi}(x+u)\,d\mu(u).
\]

\section{Historical remarks and main results}

The Walsh system in the Kaczmarz enumeration has been studied extensively; see,
for example, \cite{G1,GG,N1,Sch,Si,Sk,Sn,Yo}.  The Dirichlet kernels of the
Walsh--Kaczmarz system have a less favourable behaviour than those of the
Walsh--Paley system.

Indeed, \v{S}ne\u{\i}der, who introduced this system in 1948, proved in
\cite{Sn} that
\[
	\limsup_{n\to\infty}\frac{D_n^{\kappa}(x)}{\log n}>0
\]
for almost every $x\in G$.  In the Walsh--Paley case one has the stronger pointwise estimate
\[
	\frac{D_n^w(x)}{\log n}\to0
\]
as $n\to\infty$, for every $x\in G\setminus\{0\}$.
This growth of the Walsh--Kaczmarz Dirichlet kernels allows one to construct
functions whose Walsh--Kaczmarz Fourier series diverge almost everywhere; see
Balashov \cite{Bal} and Polyakov \cite{Pol}.

Several positive convergence results are also known.
Schipp \cite{Sch} and W. S. Young \cite{Yo} showed that the Walsh--Kaczmarz
system is a convergence system.  Skvortsov \cite{Sk} proved uniform convergence
of the Fej\'er means to $f$ for every continuous function $f$.  G\'at \cite{G1}
proved almost everywhere convergence of the Fej\'er means for every
$f\in L_1(G)$.  Also, as in the Walsh--Paley case, the partial sums
$S_n^{\kappa}(f)$ converge to $f$ in $L_p(G)$ whenever $1<p<\infty$.

The $L_1$-behaviour of Fej\'er kernels associated with the Walsh--Paley system is
a classical topic in summability theory.  Early results gave
\[
\sup_{n\in\mathbb P}\|K_n^w\|_1<2
\]
(see, e.g., Yano \cite{Y1}).  Toledo \cite[Theorem~7]{Tol} improved this and
proved the exact supremum
\[
\sup_{n\in\mathbb P}\|K_n^w\|_1=\frac{17}{15}.
\]
At powers of two, the Walsh--Paley case is especially simple:
\[
\|K_{2^n}^w\|_1=1\qquad(n\in\mathbb N).
\]

In the Walsh--Kaczmarz case sharp $L_1(G)$-bounds for the Fej\'er kernels are
more difficult.  Skvortsov's decomposition relates the Walsh--Kaczmarz kernels
to the Walsh--Paley kernels.  However, a sharp estimate for the $L_1$-norms does
not follow from a direct triangle inequality.  Blahota and D. Nagy \cite{BD3}
obtained the uniform estimate
\begin{equation}\label{BDt}
	\|K_n^{\kappa}\|_1<\frac{32}{15}\qquad(n\in\mathbb P).
\end{equation}
This estimate is not sharp for the present problem.

Our first aim is to determine the exact supremum over indices that are powers of
two.  We prove
\[
\sup_{n\in\mathbb N}\|K_{2^n}^{\kappa}\|_1
=
\lim_{n\to\infty}\|K_{2^n}^{\kappa}\|_1
=
\frac43.
\]
We also prove that the sequence $(\|K_{2^n}^{\kappa}\|_1:n\in\mathbb N)$ is
non-decreasing and is strictly increasing from $n=2$ on.

Our second aim is to solve the global supremum problem.  We combine the analysis
along powers of two with an exact max identity from Skvortsov's decomposition
and with Toledo's blockwise Walsh--Paley estimates \cite[Theorems~2 and~6]{Tol}.
This gives the sharp result
\[
\sup_{n\in\mathbb P}\|K_n^{\kappa}\|_1=\frac{71}{50}.
\]
The value $71/50$ is approached along the indices
\[
2^m+\left\lceil\frac{2^{m+1}}3\right\rceil,
\]
which are closely related to Toledo's alternating extremal indices
\cite[Theorem~6]{Tol}.

\section{Auxiliary results}

\begin{lemma}[Paley's lemma \cite{SWSP}, p.~7]\label{Paley}
Let $n\in\mathbb N$.  Then
\[
	D_{2^n}(x):=D_{2^n}^w(x)=D_{2^n}^{\kappa}(x)=
	\begin{cases}
		2^n, & x\in I_n,\\
		0, & x\notin I_n.
	\end{cases}
\]
\end{lemma}

\begin{lemma}[Toledo's Walsh--Paley estimates]\label{L:Toledo-Paley-estimates}\label{L:blockwise-Toledo-excess}
The Walsh--Paley Fej\'er kernels satisfy
\begin{equation}\label{eq:Toledo-global}
	\sup_{n\in\mathbb P}\|K_n^w\|_1=\frac{17}{15}.
\end{equation}
Moreover, for every $n\in\mathbb P$,
\begin{equation}\label{eq:Toledo-block-excess}
	n\|K_n^w\|_1-n\le\frac{8}{45}2^{|n|}.
\end{equation}
\end{lemma}

\begin{proof}
The identity \eqref{eq:Toledo-global} is Toledo's main theorem; see
\cite[Theorem~7]{Tol}.

Put
\[
	E_n:=n\|K_n^w\|_1-n.
\]
We use Toledo's explicit iteration formula, in particular
\cite[Theorem~2 and formula~(8)]{Tol}.  We also use the block comparison and
symmetry results which lead to \cite[Theorems~4--6]{Tol}.  These results show
that, in the dyadic block
\[
	2^m\le n<2^{m+1},
\]
the maximum of $E_n$ is attained at Toledo's alternating extremal index
\[
	N_m:=
	\begin{cases}
		1+2^2+\cdots+2^m, & m \text{ even},\\
		2+2^3+\cdots+2^m, & m \text{ odd},
	\end{cases}
\]
or at its symmetric partner with respect to the relation
$n+\widetilde n=3\cdot2^m-1$.  The two indices give the same value of $E_n$ by
\cite[Theorem~4]{Tol}.

It is enough to evaluate Toledo's explicit formula at $N_m$.  If $m=2s$, then
\[
	N_m=\frac{4^{s+1}-1}{3}
\]
and Toledo's formula gives
\[
	\|K_{N_m}^w\|_1
	=
	\frac{17}{15}
	-
	\frac{s+1}{4^{s+1}-1}
	+
	\frac1{5\cdot4^s}.
\]
Hence
\[
\begin{aligned}
	E_{N_m}
	&=
	N_m\left(\|K_{N_m}^w\|_1-1\right)  \\
	&=
	\frac{8}{45}4^s
	-
	\frac{3s+1}{9}
	-
	\frac1{15\cdot4^s}
	\le
	\frac{8}{45}4^s
	=
	\frac{8}{45}2^m.
\end{aligned}
\]
If $m=2s+1$, then
\[
	N_m=\frac{2(4^{s+1}-1)}{3}
\]
and Toledo's formula gives
\[
	\|K_{N_m}^w\|_1
	=
	\frac{17}{15}
	-
	\frac12\frac{s+1}{4^{s+1}-1}
	+
	\frac1{30\cdot4^s}.
\]
Therefore
\[
\begin{aligned}
	E_{N_m}
	&=
	N_m\left(\|K_{N_m}^w\|_1-1\right)  \\
	&=
	\frac{16}{45}4^s
	-
	\frac{s+1}{3}
	-
	\frac1{45\cdot4^s}
	\le
	\frac{16}{45}4^s
	=
	\frac{8}{45}2^m.
\end{aligned}
\]
This proves \eqref{eq:Toledo-block-excess} in the whole block.
\end{proof}

\begin{remark}\label{rem:Paley-dyadic}
It is known that $\|K_{2^n}^w\|_1=1$ for every $n\in\mathbb N$.
\end{remark}

\begin{lemma}[Skvortsov \cite{Sk}]\label{Skv}
Let $n\in\mathbb P$ and write $n=2^{|n|}+k$, where
$k\in\{0,\dots,2^{|n|}-1\}$.  Then
\[
	nK_n^{\kappa}(x)
	=1+\sum_{i=0}^{|n|-1}2^iD_{2^i}(x)
	+\sum_{i=0}^{|n|-1}2^ir_i(x)K_{2^i}^w(\tau_i(x))+kD_{2^{|n|}}(x)
	+kr_{|n|}(x)K_k^w(\tau_{|n|}(x)),
\]
where the last term is omitted if $k=0$.
\end{lemma}

\begin{corollary}\label{Skv_dyadic}
Let $n\in\mathbb N$.  Then, for every $x\in G$,
\begin{equation}\label{eq:Skv_dyadic}
	2^nK_{2^n}^{\kappa}(x)
	=
	1+\sum_{i=0}^{n-1}2^iD_{2^i}(x)
	+\sum_{i=0}^{n-1}2^ir_i(x)K_{2^i}^w(\tau_i(x)).
\end{equation}
\end{corollary}

\begin{corollary}\label{sec}
Let $n=2^{|n|}+k$, where $k\in\{0,\dots,2^{|n|}-1\}$.  Then
\[
	nK_n^{\kappa}(x)
	=
	2^{|n|}K_{2^{|n|}}^{\kappa}(x)
	+kD_{2^{|n|}}(x)
	+kr_{|n|}(x)K_k^w(\tau_{|n|}(x)),
\]
where the last term is omitted if $k=0$.
\end{corollary}

\begin{proof}
This follows from Lemma~\ref{Skv} and Corollary~\ref{Skv_dyadic}.  The first
three terms in Lemma~\ref{Skv} are exactly
$2^{|n|}K_{2^{|n|}}^{\kappa}(x)$.
\end{proof}

We write $e_n:=(0,\dots,0,1,0,\dots)\in G$, where the $n$th coordinate is the
only non-zero coordinate.

\begin{lemma}[Golubov, Skvortsov and Efimov \cite{GAS}]\label{Kw2i_atoms}
Let $n\in\mathbb N$.  Then, for every $x\in G$,
\[
K_{2^n}^w(x)=
\begin{cases}
	\dfrac{2^n+1}{2}, & x\in I_n,\\[1mm]
	2^{t-1}, & x\in I_n(e_t),\quad t=0,1,\dots,n-1,\\
	0, & \text{otherwise.}
\end{cases}
\]
\end{lemma}

\begin{remark}\label{rem}
$K_1^w(x)=K_1^{\kappa}(x)=1$ for all $x\in G$.
\end{remark}

\section[The supremum along powers of two]
{The supremum along powers of two: \texorpdfstring{$\sup_{n\in\mathbb N}\|K_{2^n}^{\kappa}\|_1$}{supremum of the L1-norms along powers of two}}

\subsection[The sharp estimate along powers of two]
{The sharp estimate for \texorpdfstring{$\|K_{2^n}^{\kappa}\|_1$}{the L1-norm along powers of two}}

Define, for $0\le s\le n-1$,
\[
	R_{s,n}(x):=\sum_{i=s+1}^{n-1}2^ir_i(x)K_{2^i}^w(\tau_i(x)),
\]
\[
	M_s(x):=1+\sum_{i=0}^{s}2^iD_{2^i}(x)
	+\sum_{i=0}^{s}2^ir_i(x)K_{2^i}^w(\tau_i(x)),
\]
and
\[
	T_n:=\sum_{s=0}^{n-1}\frac1{2^n}\int_{J_s}|R_{s,n}(x)|\,d\mu(x).
\]
By Corollary~\ref{Skv_dyadic} and Lemma~\ref{Paley},
\begin{equation}\label{ce}
	2^nK_{2^n}^{\kappa}(x)=M_s(x)+R_{s,n}(x)
\end{equation}
for every $s\in\{0,\dots,n-1\}$ and every $x\in J_s$.

\begin{lemma}\label{main_on_Es}
Let $s\in\mathbb N$.  Then, for every $x\in J_s$,
\[
	M_s(x)=4^s.
\]
\end{lemma}

\begin{proof}
Fix $x\in J_s$.  Then $x_s=1$ and, by Lemma~\ref{Paley},
$D_{2^i}(x)=2^i$ for $i\le s$.  Hence
\[
	\sum_{i=0}^{s}2^iD_{2^i}(x)=\sum_{i=0}^{s}4^i
	=\frac{4^{s+1}-1}{3}.
\]
For $i\le s$ we have $\tau_i(x)\in I_i$.  Thus, by Lemma~\ref{Kw2i_atoms},
\[
	K_{2^i}^w(\tau_i(x))=\frac{2^i+1}{2}.
\]
Moreover, $r_i(x)=1$ for $i\le s-1$ and $r_s(x)=-1$.  Thus
\[
\begin{aligned}
	\sum_{i=0}^{s}2^ir_i(x)K_{2^i}^w(\tau_i(x))
	&=\sum_{i=0}^{s-1}2^i\frac{2^i+1}{2}
	-2^s\frac{2^s+1}{2} \\
	&=-\frac{4^s+2}{3}.
\end{aligned}
\]
A simplification gives $M_s(x)=4^s$.
\end{proof}

\begin{lemma}\label{tail}
Let $n\ge2$ and let $s\in\{0,\dots,n-1\}$.

\medskip
\noindent\textup{(a)} We have
\[
	J_s=
	\left(\bigcup_{j=s+1}^{n-1}I_{j+1}(e_s+e_j)\right)\cup I_n(e_s),
\]
\[
	\mu(I_{j+1}(e_s+e_j))=2^{-j-1},
	\qquad
	\mu(I_n(e_s))=2^{-n},
\]
and these sets are pairwise disjoint.  Moreover, $R_{s,n}$ is a negative
constant on $I_{j+1}(e_s+e_j)$ and a non-negative constant on $I_n(e_s)$.

\medskip
\noindent\textup{(b)} For every $s=0,\dots,n-1$,
\[
	\int_{J_s}|R_{s,n}(x)|\,d\mu(x)
	=
	\frac{2^{n-s-1}-2^{s-n+1}}{3}.
\]
Consequently,
\[
	T_n=\frac{(2^n-2)(2^n-1)}{3\cdot4^n}.
\]
\end{lemma}

\begin{proof}
\noindent\textup{(a)} The case $s=n-1$ is trivial, since the union is empty and
$R_{n-1,n}=0$.  Fix $s\in\{0,\dots,n-2\}$.  The decomposition and the measure
formulas follow by fixing coordinates.

Fix $x\in I_{j+1}(e_s+e_j)$, where $j\in\{s+1,\dots,n-1\}$.  For
$i\in\{s+1,\dots,n-1\}$, Lemma~\ref{Kw2i_atoms} and the definition of $\tau_i$
show that $K_{2^i}^w(\tau_i(x))\ne0$ if and only if
\[
	x_{s+1}=\cdots=x_{i-1}=0.
\]
On $I_{j+1}(e_s+e_j)$ this holds exactly for $i\le j$.  Moreover, for $i\le j$,
we have $\tau_i(x)\in I_i(e_{i-1-s})$, and hence
\[
	K_{2^i}^w(\tau_i(x))=2^{i-s-2}.
\]
Finally, on $I_{j+1}(e_s+e_j)$ we have $r_i(x)=1$ for $i<j$ and $r_j(x)=-1$.
Therefore
\begin{equation}\label{R1}
	R_{s,n}(x)
	=\sum_{i=s+1}^{j-1}2^i2^{i-s-2}-2^j2^{j-s-2}
	=-\frac{2^{2j-s-1}+2^s}{3}<0.
\end{equation}

If $x\in I_n(e_s)$, then $x_{s+1}=\cdots=x_{n-1}=0$.  Thus, for every
$i\in\{s+1,\dots,n-1\}$,
$\tau_i(x)\in I_i(e_{i-1-s})$ and $K_{2^i}^w(\tau_i(x))=2^{i-s-2}$, while
$r_i(x)=1$.  Hence
\begin{equation}\label{R2}
	R_{s,n}(x)
	=\sum_{i=s+1}^{n-1}2^i2^{i-s-2}
	=\frac{2^{2n-s-2}-2^s}{3}\ge0.
\end{equation}
Since $R_{s,n}$ depends only on the first $n$ coordinates, it is constant on
each of these sets.

\medskip
\noindent\textup{(b)} The identity is trivial for $s=n-1$.  For
$s\le n-2$, the sign information in part (a), together with \eqref{R1} and
\eqref{R2}, gives
\[
\begin{aligned}
	\int_{J_s}|R_{s,n}|\,d\mu
	&=\sum_{j=s+1}^{n-1}2^{-j-1}\frac{2^{2j-s-1}+2^s}{3}
	+2^{-n}\frac{2^{2n-s-2}-2^s}{3} \\
	&=\frac{2^{n-s-1}-2^{s-n+1}}{3}.
\end{aligned}
\]
Summing over $s$ gives the formula for $T_n$.
\end{proof}

\begin{lemma}\label{main_mass_identity}
Let $n\in\mathbb N$.  Then
\[
	\int_{I_n}|K_{2^n}^{\kappa}|\,d\mu=\frac12+\frac1{2^{n+1}}.
\]
\end{lemma}

\begin{proof}
The case $n=0$ follows from Remark~\ref{rem}.  Let $n\ge1$.  If $x\in I_n$,
then $x_0=\cdots=x_{n-1}=0$.  Thus, for $i\le n-1$, Lemma~\ref{Paley} gives
$D_{2^i}(x)=2^i$, while $r_i(x)=1$ and $\tau_i(x)\in I_i$.  By
Lemma~\ref{Kw2i_atoms}, $K_{2^i}^w(\tau_i(x))=(2^i+1)/2$.  Substitution in
\eqref{eq:Skv_dyadic} yields
\[
	2^nK_{2^n}^{\kappa}(x)
	=1+\sum_{i=0}^{n-1}4^i+\frac12\sum_{i=0}^{n-1}(4^i+2^i)
	=2^{2n-1}+2^{n-1}.
\]
Therefore $K_{2^n}^{\kappa}(x)=2^{n-1}+1/2$ on $I_n$, and hence
\[
	\int_{I_n}|K_{2^n}^{\kappa}|\,d\mu
	=\left(2^{n-1}+\frac12\right)2^{-n}
	=\frac12+\frac1{2^{n+1}}.
\]
\end{proof}

Define
\[
\mathcal E_n:=
\frac2{2^n}
\sum_{s=0}^{n-1}\sum_{j=s+1}^{n-1}
\mu(I_{j+1}(e_s+e_j))
\min\left\{4^s,-R_{s,n}|_{I_{j+1}(e_s+e_j)}\right\}.
\]

\begin{lemma}\label{Em_estimate}
For all $n\ge2$,
\[
	0\le\mathcal E_n\le\frac{7n}{2^{2n/3}}.
\]
\end{lemma}

\begin{proof}
The inequality $\mathcal E_n\ge0$ is immediate.  By Lemma~\ref{tail}(a) and
\eqref{R1},
\[
	\mu(I_{j+1}(e_s+e_j))=2^{-j-1},
	\qquad
	-R_{s,n}|_{I_{j+1}(e_s+e_j)}=\frac{2^{2j-s-1}+2^s}{3}.
\]
Since $\min\{a,b+c\}\le\min\{a,b\}+c$ for $a,b,c\ge0$,
\[
\begin{aligned}
	\min\left\{4^s,\frac{2^{2j-s-1}+2^s}{3}\right\}
	&\le
	\min\left\{2^{2s},\frac{2^{2j-s}}6\right\}+\frac{2^s}{3} \\
	&\le \min\{2^{2s},2^{2j-s}\}+\frac{2^s}{3}.
\end{aligned}
\]
Since
\[
	\frac2{2^n}\mu(I_{j+1}(e_s+e_j))
	=\frac2{2^n}2^{-j-1}
	=\frac1{2^n}2^{-j},
\]
we obtain $\mathcal E_n\le \Sigma_{1,n}+\Sigma_{2,n}$, where
\[
	\Sigma_{1,n}:=\frac1{2^n}\sum_{s=0}^{n-1}\sum_{j=s+1}^{n-1}
	\frac{\min\{2^{2s},2^{2j-s}\}}{2^j},
\]
and
\[
	\Sigma_{2,n}:=\frac1{3\cdot2^n}\sum_{s=0}^{n-1}2^s
	\sum_{j=s+1}^{n-1}2^{-j}.
\]
For $\Sigma_{2,n}$ we use $\sum_{j=s+1}^{n-1}2^{-j}\le2^{-s}$ and get
\[
	\Sigma_{2,n}\le\frac{n}{3\cdot2^n}\le\frac{n}{2^{2n/3}}.
\]

To estimate $\Sigma_{1,n}$, split the inner sum at $\lfloor3s/2\rfloor$.  If
$s\le2(n-1)/3$, then
\[
	\sum_{j=s+1}^{n-1}\frac{\min\{2^{2s},2^{2j-s}\}}{2^j}
	\le
	\sum_{j=s+1}^{\lfloor3s/2\rfloor}2^{j-s}
	+
	\sum_{j=\lfloor3s/2\rfloor+1}^{n-1}2^{2s-j}
	\le 3\cdot2^{s/2}.
\]
If $s>2(n-1)/3$, then
\[
	\sum_{j=s+1}^{n-1}\frac{\min\{2^{2s},2^{2j-s}\}}{2^j}
	=\sum_{j=s+1}^{n-1}2^{j-s}
	\le2^{n-s}.
\]
Consequently,
\[
	\Sigma_{1,n}\le\frac{12}{2^{2n/3}}.
\]
Combining the estimates gives
\[
	\mathcal E_n\le\frac{12+n}{2^{2n/3}}\le\frac{7n}{2^{2n/3}}
\]
for $n\ge2$.
\end{proof}

\begin{theorem}\label{T:exact_upper_limit}
Let $n\ge2$.  Then
\[
	\|K_{2^n}^{\kappa}\|_1=1+T_n-\mathcal E_n.
\]
Consequently,
\[
	\|K_{2^n}^{\kappa}\|_1\le1+T_n<\frac43.
\]
Moreover,
\[
	\lim_{n\to\infty}\|K_{2^n}^{\kappa}\|_1=\frac43.
\]
\end{theorem}

\begin{proof}
By \eqref{ce} and Lemma~\ref{main_on_Es},
\[
	\int_{J_s}|K_{2^n}^{\kappa}|\,d\mu
	=\frac1{2^n}\int_{J_s}|4^s+R_{s,n}(x)|\,d\mu(x).
\]
Set
\[
	\Delta_{s,n}:=\int_{J_s}\left(4^s+|R_{s,n}(x)|-|4^s+R_{s,n}(x)|\right)d\mu(x).
\]
Then
\[
	\int_{J_s}|K_{2^n}^{\kappa}|\,d\mu
	=\frac1{2^n}\left(4^s\mu(J_s)+\int_{J_s}|R_{s,n}|\,d\mu-\Delta_{s,n}\right).
\]
The decomposition
\[
	G=I_n\cup\left(\bigcup_{s=0}^{n-1}J_s\right)
\]
is pairwise disjoint, and
\[
	\frac12+\frac1{2^{n+1}}+
	\sum_{s=0}^{n-1}\frac1{2^n}4^s\mu(J_s)=1.
\]
Using Lemma~\ref{main_mass_identity}, we get
\[
	\|K_{2^n}^{\kappa}\|_1
	=1+\sum_{s=0}^{n-1}\frac1{2^n}\int_{J_s}|R_{s,n}|\,d\mu
	-\frac1{2^n}\sum_{s=0}^{n-1}\Delta_{s,n}.
\]
For $a>0$ and $u\in\mathbb R$,
\[
	a+|u|-|a+u|=
	\begin{cases}
		0, & u\ge0,\\
		2\min\{a,|u|\}, & u<0.
	\end{cases}
\]
By Lemma~\ref{tail}(a), this gives
\[
	\Delta_{s,n}=2\sum_{j=s+1}^{n-1}\mu(I_{j+1}(e_s+e_j))
	\min\left\{4^s,-R_{s,n}|_{I_{j+1}(e_s+e_j)}\right\}.
\]
Hence
\[
	\frac1{2^n}\sum_{s=0}^{n-1}\Delta_{s,n}=\mathcal E_n,
\]
and therefore $\|K_{2^n}^{\kappa}\|_1=1+T_n-\mathcal E_n$.

By Lemma~\ref{tail}(b),
\[
	T_n=\frac{(2^n-2)(2^n-1)}{3\cdot4^n}<\frac13.
\]
Since $\mathcal E_n\ge0$, we obtain $\|K_{2^n}^{\kappa}\|_1<4/3$.  Finally,
\[
	1+T_n=\frac43-\frac1{2^n}+\frac2{3\cdot4^n}\to\frac43
	\qquad(n\to\infty),
\]
and Lemma~\ref{Em_estimate} gives $\mathcal E_n\to0$ as $n\to\infty$.  Thus
$\|K_{2^n}^{\kappa}\|_1\to4/3$ as $n\to\infty$.
\end{proof}

\begin{corollary}\label{third}
We have
\[
	\sup_{n\in\mathbb N}\|K_{2^n}^{\kappa}\|_1=\frac43.
\]
\end{corollary}

\begin{proof}
For $n=0,1$ we have $\|K_{2^n}^{\kappa}\|_1=1$.  For $n\ge2$ the assertion
follows from Theorem~\ref{T:exact_upper_limit}.
\end{proof}

\subsection[Monotonicity along powers of two]
{Monotonicity of \texorpdfstring{$\|K_{2^n}^{\kappa}\|_1$}{the L1-norm along powers of two}}

In this section we prove that $(\|K_{2^n}^{\kappa}\|_1:n\in\mathbb N)$ is
non-decreasing and that it is strictly increasing from $n=2$ on.

\begin{lemma}\label{lem:correction_q}
For $j\in\mathbb P$ define
\[
	q_j:=\sum_{s=0}^{j-1}2^{-j}
	\min\left\{2^{2s},\frac{2^{2j-s-1}+2^s}{3}\right\}.
\]
Then, for every $n\ge2$,
\[
	\mathcal E_n=2^{-n}\sum_{j=1}^{n-1}q_j.
\]
Moreover,
\[
	q_j=\frac{2^{r_j}-2^{-r_j}+2^{j-2r_j}-2^{-j}}3,
	\qquad
	r_j:=\left\lceil\frac j3\right\rceil.
\]
\end{lemma}

\begin{proof}
Using the definition of $\mathcal E_n$ and \eqref{R1},
\[
\begin{aligned}
\mathcal E_n
&=\frac2{2^n}\sum_{s=0}^{n-1}\sum_{j=s+1}^{n-1}
2^{-j-1}\min\left\{4^s,\frac{2^{2j-s-1}+2^s}{3}\right\} \\
&=2^{-n}\sum_{j=1}^{n-1}\sum_{s=0}^{j-1}2^{-j}
\min\left\{2^{2s},\frac{2^{2j-s-1}+2^s}{3}\right\}
=2^{-n}\sum_{j=1}^{n-1}q_j.
\end{aligned}
\]

Put $d:=j-s$.  Then
\[
	q_j=\sum_{d=1}^{j}\min\left\{2^{j-2d},\frac{2^{d-1}+2^{-d}}3\right\}.
\]
The inequality
\[
	\frac{2^{d-1}+2^{-d}}3\le2^{j-2d}
\]
is equivalent to
\[
	2^{3d-1}+2^d\le3\cdot2^j.
\]
The left-hand side is increasing in $d$.  A check according to the residue class
of $j$ modulo $3$ shows that the inequality holds exactly when
$d\le r_j=\lceil j/3\rceil$.  Hence
\[
	q_j=
	\sum_{d=1}^{r_j}\frac{2^{d-1}+2^{-d}}3
	+
	\sum_{d=r_j+1}^{j}2^{j-2d}.
\]
Both sums are geometric.  This gives the formula.
\end{proof}

\begin{lemma}\label{lem:q_subadditive}
For every $j\ge3$,
\[
	q_j\le q_{j-1}+q_{j-2}.
\]
\end{lemma}

\begin{proof}
By Lemma~\ref{lem:correction_q}, according to the residue class of $j$ modulo
$3$,
\[
	q_{3r}=\frac{2^{r+1}-2^{-r}-2^{-3r}}3\qquad(r\ge1),
\]
\[
	q_{3r+1}=\frac{5\cdot2^{r-1}-2^{-r-1}-2^{-3r-1}}3\qquad(r\ge0),
\]
and
\[
	q_{3r+2}=\frac{3\cdot2^r-2^{-r-1}-2^{-3r-2}}3\qquad(r\ge0).
\]
If $j=3r$, then
\[
	q_{3r-1}+q_{3r-2}-q_{3r}
	=2^{r-2}-\frac1{3\cdot2^r}-\frac5{3\cdot2^{3r}}>0.
\]
If $j=3r+1$, then
\[
	q_{3r}+q_{3r-1}-q_{3r+1}
	=\frac{2^r}{3}-\frac1{2\cdot2^r}-\frac5{6\cdot2^{3r}}>0.
\]
If $j=3r+2$, then
\[
	q_{3r+1}+q_{3r}-q_{3r+2}
	=2^{r-1}-\frac1{3\cdot2^r}-\frac5{12\cdot2^{3r}}>0.
\]
This proves the lemma.
\end{proof}

\begin{theorem}\label{T:dyadic_monotonicity}
The sequence
\[
	(\|K_{2^n}^{\kappa}\|_1:n\in\mathbb N)
\]
is non-decreasing.  More precisely,
\[
	\|K_{2^0}^{\kappa}\|_1
	=\|K_{2^1}^{\kappa}\|_1
	=\|K_{2^2}^{\kappa}\|_1=1,
\]
and
\[
	\|K_{2^{n+1}}^{\kappa}\|_1>\|K_{2^n}^{\kappa}\|_1
	\qquad(n\ge2).
\]
\end{theorem}

\begin{proof}
By Theorem~\ref{T:exact_upper_limit}, for $n\ge2$,
\[
	\|K_{2^n}^{\kappa}\|_1=1+T_n-\mathcal E_n,
	\qquad
	T_n=\frac{(2^n-2)(2^n-1)}{3\cdot4^n}.
\]
A direct calculation gives
\[
	T_{n+1}-T_n=\frac1{2^{n+1}}\left(1-\frac1{2^n}\right).
\]
By Lemma~\ref{lem:correction_q},
\[
	\mathcal E_n=2^{-n}\sum_{j=1}^{n-1}q_j,
\]
and hence
\[
	\mathcal E_{n+1}-\mathcal E_n
	=2^{-(n+1)}\left(q_n-\sum_{j=1}^{n-1}q_j\right).
\]
Therefore
\[
\begin{aligned}
	\|K_{2^{n+1}}^{\kappa}\|_1-\|K_{2^n}^{\kappa}\|_1
	=\frac1{2^{n+1}}\left(1-\frac1{2^n}+\sum_{j=1}^{n-1}q_j-q_n\right).
\end{aligned}
\]
For $n\ge3$, Lemma~\ref{lem:q_subadditive} gives
$q_n\le q_{n-1}+q_{n-2}\le\sum_{j=1}^{n-1}q_j$.  Thus
\[
	\|K_{2^{n+1}}^{\kappa}\|_1-\|K_{2^n}^{\kappa}\|_1
	\ge\frac1{2^{n+1}}\left(1-\frac1{2^n}\right)>0.
\]
The initial values follow from $\|K_1^{\kappa}\|_1=\|K_2^{\kappa}\|_1=1$ and from
$q_1=1/2$, $T_2=1/8$, $\mathcal E_2=2^{-2}q_1=1/8$.  Also
$q_2=3/4$, $T_3=7/32$, and $\mathcal E_3=2^{-3}(q_1+q_2)=5/32$, so
\[
	\|K_{2^3}^{\kappa}\|_1=1+\frac7{32}-\frac5{32}=\frac{17}{16}>1.
\]
\end{proof}

The first values of the sequence along powers of two are listed in Table~\ref{L1-first22}. They agree with the monotonicity proved in Theorem~\ref{T:dyadic_monotonicity} and illustrate the convergence to the sharp supremum along powers of two, namely $4/3$.
\begin{table}[ht]
\centering
\setlength{\tabcolsep}{14pt}
\renewcommand{\arraystretch}{1.25}
\begin{tabular}{r@{\qquad}r@{\qquad}c@{\qquad}c@{\qquad\qquad}r@{\qquad}r@{\qquad}c@{\qquad}c}
\cmidrule(l{1.2em}r{4em}){1-4}\cmidrule(l{0em}r{1.2em}){5-8}
$n$ & $2^n$ & $\|K^{\kappa}_{2^n}\|_1$ & $\approx$
& $n$ & $2^n$ & $\|K^{\kappa}_{2^n}\|_1$ & $\approx$\\
\cmidrule(l{1.2em}r{4em}){1-4}\cmidrule(l{0em}r{1.2em}){5-8}
0  & 1       & $1$                      & 1.000000  & 11 & 2048    & $\tfrac{43233}{32768}$     & 1.319366 \\
1  & 2       & $1$                      & 1.000000  & 12 & 4096    & $\tfrac{21699}{16384}$     & 1.324402 \\
2  & 4       & $1$                      & 1.000000  & 13 & 8192    & $\tfrac{174007}{131072}$   & 1.327568 \\
3  & 8       & $\tfrac{17}{16}$         & 1.062500  & 14 & 16384   & $\tfrac{697113}{524288}$   & 1.329638 \\
4  & 16      & $\tfrac{9}{8}$           & 1.125000  & 15 & 32768   & $\tfrac{348913}{262144}$   & 1.330997 \\
5  & 32      & $\tfrac{151}{128}$       & 1.179688  & 16 & 65536   & $\tfrac{2793071}{2097152}$ & 1.331840 \\
6  & 64      & $\tfrac{157}{128}$       & 1.226562  & 17 & 131072  & $\tfrac{11176841}{8388608}$& 1.332383 \\
7  & 128     & $\tfrac{645}{512}$       & 1.259766  & 18 & 262144  & $\tfrac{5589901}{4194304}$ & 1.332736 \\
8  & 256     & $\tfrac{2629}{2048}$     & 1.283691  & 19 & 524288  & $\tfrac{44726495}{33554432}$& 1.332953 \\
9  & 512     & $\tfrac{333}{256}$       & 1.300781  & 20 & 1048576 & $\tfrac{178924649}{134217728}$& 1.333093 \\
10 & 1024    & $\tfrac{10747}{8192}$    & 1.311890 & 21 & 2097152 & $\tfrac{89468357}{67108864}$& 1.333182 \\
\cmidrule(l{1.2em}r{4em}){1-4}\cmidrule(l{0em}r{1.2em}){5-8}
\end{tabular}
\caption{The first 22 terms of the sequence $(\|K_{2^n}^{\kappa}\|_1:n\in\mathbb N)$.}
\label{L1-first22}
\end{table}

\section[The global supremum]{The global supremum \texorpdfstring{$\sup_{n\in\mathbb P}\|K_n^{\kappa}\|_1$}{the global supremum of the L1-norms}}

The numerical behaviour of the full sequence
$(\|K_n^{\kappa}\|_1:n\in\mathbb P)$ for the first values of $n$ is shown in Figure~\ref{L1-32768}. The figure is
included only as a numerical illustration; the exact value of the global supremum is proved below.

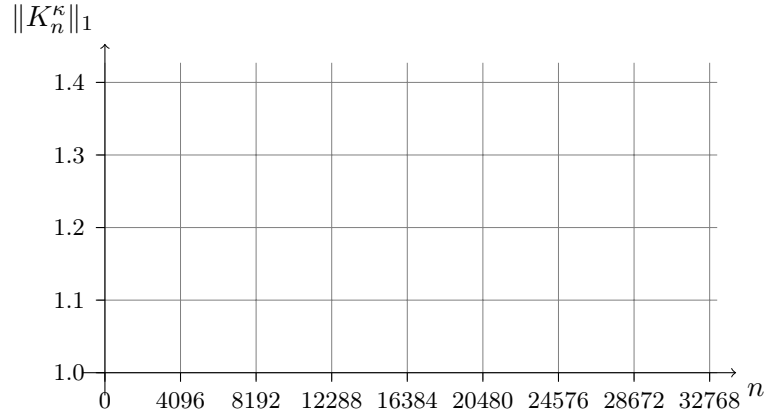
\begin{figure}[ht]
\centering
\begin{tikzpicture}[font=\small]
\draw[ultra thin,color=gray] (-0.1,-0.1) grid[xstep=1, ystep=0.960077915605] (8.1,4.1);
\draw[->] (-0.3,0) -- (8.35,0) node[below right] {$n$};
\draw[->] (0,-0.3) -- (0,4.35) node[above left] {$\|K^{\kappa}_n\|_{1}$};

\foreach \x/\lab in {0/0,1/4096,2/8192,3/12288,4/16384,5/20480,6/24576,7/28672,8/32768} {
	\draw (\x,0) -- (\x,-0.12);
	\node[below] at (\x,-0.12) {\scriptsize \textnormal{\lab}};
}

\draw (-0.12,0.000000000000) -- (0,0.000000000000);
\node[left] at (-0.12,0.000000000000) {\scriptsize \textnormal{1.0}};
\draw (-0.12,0.960077915605) -- (0,0.960077915605);
\node[left] at (-0.12,0.960077915605) {\scriptsize \textnormal{1.1}};
\draw (-0.12,1.920155831210) -- (0,1.920155831210);
\node[left] at (-0.12,1.920155831210) {\scriptsize \textnormal{1.2}};
\draw (-0.12,2.880233746814) -- (0,2.880233746814);
\node[left] at (-0.12,2.880233746814) {\scriptsize \textnormal{1.3}};
\draw (-0.12,3.840311662419) -- (0,3.840311662419);
\node[left] at (-0.12,3.840311662419) {\scriptsize \textnormal{1.4}};

\InputIfFileExists{data.tex}{}{}
\end{tikzpicture}
\caption{Numerically computed values of $\|K^{\kappa}_{n}\|_{1}$ for $n=1,\dots,32768$.}
\label{L1-32768}
\end{figure}

\subsection{Exact splitting for general indices}

For $m\in\mathbb N$ and $0\le v\le m-1$, define
\begin{equation}\label{eq:Qml-def}
	Q_{m, v}:=4^{ v}+\frac{2^{2m- v-2}-2^{ v}}3.
\end{equation}
The analysis above gives the following atom values.  On $I_m$,
\begin{equation}\label{eq:F-on-Im}
	2^mK_{2^m}^{\kappa}(x)=2^{2m-1}+2^{m-1},
\end{equation}
while, for $0\le v\le m-1$,
\begin{equation}\label{eq:F-on-singletons}
	2^mK_{2^m}^{\kappa}(x)=Q_{m, v},
	\qquad x\in I_m(e_ v).
\end{equation}
For $0\le s<j\le m-1$,
\begin{equation}\label{eq:F-on-Esj}
	2^mK_{2^m}^{\kappa}(x)
	=4^s-\frac{2^{2j-s-1}+2^s}{3},
	\qquad x\in I_{j+1}(e_s+e_j).
\end{equation}
The first identity is Lemma~\ref{main_mass_identity} at the pointwise level.
The last two follow from Lemma~\ref{main_on_Es} and formulas
\eqref{R1}--\eqref{R2}.

Throughout this subsection, if $n\in\mathbb P$ is not a power of two, we write
\[
	k:=n-2^{|n|}.
\]
Then $1\le k<2^{|n|}$, and Corollary~\ref{sec} gives
\begin{equation}\label{eq:Skvortsov-split}
	nK_n^{\kappa}(x)
	=
	2^{|n|}K_{2^{|n|}}^{\kappa}(x)
	+kD_{2^{|n|}}(x)
	+kr_{|n|}(x)K_k^w(\tau_{|n|}(x)).
\end{equation}
Since $k<2^{|n|}$, the Walsh--Paley polynomial $kK_k^w$ depends only on the
first $|n|$ coordinates.  Hence $kK_k^w(\tau_{|n|}(x))$ is independent of
$x_{|n|}$.  The same is true of
\[
	2^{|n|}K_{2^{|n|}}^{\kappa}(x)+kD_{2^{|n|}}(x).
\]
We average with respect to the coordinate $x_{|n|}$.  This gives the following
exact max identity.

\begin{lemma}\label{L:exact-max}
Let $n\in\mathbb P$ be not a power of two, and let $k$ be as above.  Then
\begin{equation}\label{eq:exact-max}
	n\|K_n^{\kappa}\|_1
	=
	\int_G
	\max\left(
	|2^{|n|}K_{2^{|n|}}^{\kappa}(x)+kD_{2^{|n|}}(x)|,
	|kK_k^w(\tau_{|n|}(x))|
	\right)d\mu(x).
\end{equation}
\end{lemma}

\begin{proof}
Fix all coordinates of $x$ except $x_{|n|}$.  Then
\[
	a:=2^{|n|}K_{2^{|n|}}^{\kappa}(x)+kD_{2^{|n|}}(x),
	\qquad
	b:=kK_k^w(\tau_{|n|}(x))
\]
are constants, while $r_{|n|}(x)=\pm1$.  Hence
\[
	\frac12(|a+b|+|a-b|)=\max(|a|,|b|).
\]
Integrating over the remaining coordinates gives \eqref{eq:exact-max}.
\end{proof}

For $m\in\mathbb N$ and $k\in\mathbb P$ with $k<2^m$, define
\begin{equation}\label{eq:def-G}
	\mathcal G_m(k):=
	\int_G
	\min\left(
	|2^mK_{2^m}^{\kappa}(x)+kD_{2^m}(x)|,
	|kK_k^w(\tau_m(x))|
	\right)d\mu(x).
\end{equation}

The exact max identity gives the following gain formula.

\begin{lemma}\label{L:gain-formula}
Let $n\in\mathbb P$ be not a power of two, and let $k$ be as above.  Then
\begin{equation}\label{eq:gain-formula}
	n\|K_n^{\kappa}\|_1
	=
	2^{|n|}\|K_{2^{|n|}}^{\kappa}\|_1
	+k
	+k\|K_k^w\|_1
	-\mathcal G_{|n|}(k).
\end{equation}
\end{lemma}

\begin{proof}
Use
\[
	\max(|A|,|B|)=|A|+|B|-\min(|A|,|B|)
\]
in \eqref{eq:exact-max}.  Since $\tau_{|n|}$ is measure-preserving,
\[
	\int_G|kK_k^w(\tau_{|n|}(x))|\,d\mu(x)=k\|K_k^w\|_1.
\]
Moreover,
\[
\begin{aligned}
	&\int_G|2^{|n|}K_{2^{|n|}}^{\kappa}(x)+kD_{2^{|n|}}(x)|\,d\mu(x) \\
	&\qquad=2^{|n|}\|K_{2^{|n|}}^{\kappa}\|_1+k.
\end{aligned}
\]
Indeed, $D_{2^{|n|}}$ is supported on $I_{|n|}$ and equals $2^{|n|}$ there.
Also, $2^{|n|}K_{2^{|n|}}^{\kappa}$ is positive on $I_{|n|}$ by
\eqref{eq:F-on-Im}.  Hence the added term contributes
$k2^{|n|}\mu(I_{|n|})=k$.  This proves \eqref{eq:gain-formula}.
\end{proof}

\subsection{The principal gain}

We keep the part of $\mathcal G_m(k)$ which comes from $I_m$ and from the
one-bit atoms $I_m(e_ v)$, $0\le v\le m-1$.

On $I_m$ one has $D_{2^m}=2^m$.  Moreover, the first $m$
coordinates of $\tau_m(x)$ are zero.  Since $k<2^m$, the polynomial
$kK_k^w$ depends only on the first $m$ coordinates, and hence
\[
kK_k^w(\tau_m(x))=kK_k^w(0)=\frac{k(k+1)}2.
\]

Moreover, by \eqref{eq:F-on-Im},
\[
	|2^mK_{2^m}^{\kappa}(x)+kD_{2^m}(x)|>|kK_k^w(0)|,
	\qquad x\in I_m.
\]
Thus the contribution of $I_m$ to $\mathcal G_m(k)$ is
\[
	\frac1{2^m}\frac{k(k+1)}2.
\]

For $0\le v\le m-1$, the function $D_{2^m}$ vanishes on $I_m(e_ v)$.  By
\eqref{eq:F-on-singletons},
\[
	2^mK_{2^m}^{\kappa}(x)=Q_{m, v},
	\qquad x\in I_m(e_ v).
\]
Furthermore,
\[
	\tau_m(I_m(e_ v))=I_m(e_{m-1- v}).
\]
Since $k<2^m$, the Walsh--Paley polynomial $kK_k^w$ depends only on the first
$m$ coordinates. Hence it is constant on each interval $I_m(a)$, $a\in G$.
Therefore
\[
kK_k^w(\tau_m(x))=kK_k^w(e_{m-1- v}),
\qquad x\in I_m(e_ v).
\]
Define the principal gain by
\begin{equation}\label{eq:def-A}
	\mathcal A_m(k)
	:=\frac1{2^m}\frac{k(k+1)}2+
	\frac1{2^m}\sum_{ v=0}^{m-1}
	\min\left(Q_{m, v},|kK_k^w(e_{m-1- v})|\right).
\end{equation}

The following lemma gives a lower bound by keeping only these principal atoms.

\begin{lemma}\label{L:principal-gain}
For every $m\in\mathbb N$ and every $k\in\mathbb P$ with $k<2^m$,
\[
	\mathcal G_m(k)\ge\mathcal A_m(k).
\]
\end{lemma}

\begin{proof}
The atoms $I_m$ and $I_m(e_ v)$, $0\le v\le m-1$, are pairwise disjoint.  On
each of them the integrand defining $\mathcal G_m(k)$ is constant and equal to
the corresponding term in \eqref{eq:def-A}.  Since the integrand is
non-negative, the sum over these atoms is a lower bound for the whole integral.
\end{proof}

\subsection{The lower sharpness sequence}

Let
\begin{equation}\label{eq:def-km}
	k_m:=\left\lceil\frac{2^{m+1}}3\right\rceil.
\end{equation}
This sequence has an asymptotically alternating binary expansion.

We first record the Walsh--Paley asymptotics along this sequence.

\begin{lemma}\label{L:alternating-Paley}
For $k_m$ given by \eqref{eq:def-km}, we have, as $m\to\infty$,
\[
	\|K_{k_m}^w\|_1\to\frac{17}{15},
	\qquad
	\frac{k_m\|K_{k_m}^w\|_1}{2^m}\to\frac{34}{45}.
\]
\end{lemma}

\begin{proof}
We use Toledo's extremal theorem for the Walsh--Paley Fej\'er kernels.  Let
\[
	a_p:=
	\begin{cases}
	1+2^2+\cdots+2^p, & p \text{ even},\\
	2+2^3+\cdots+2^p, & p \text{ odd}.
	\end{cases}
\]
By \cite[Theorem~6]{Tol}, as $p\to\infty$,
\[
	\|K_{a_p}^w\|_1\to\frac{17}{15}.
\]
For $m\ge2$, the index $a_{m-1}$ is the alternating index immediately before
$k_m$; more precisely,
\[
	k_m=a_{m-1}+1.
\]
This follows by summing the appropriate finite geometric series, according to
the parity of $m$.

It remains to note that increasing the index by one does not change the
asymptotic $L_1$-norm.  Since
\[
	K_{n+1}^w=\frac{nK_n^w+D_{n+1}^w}{n+1},
\]
we have
\[
	K_{n+1}^w-K_n^w=\frac{D_{n+1}^w-K_n^w}{n+1}.
\]
Therefore
\[
	\|K_{n+1}^w-K_n^w\|_1
	\le\frac{\|D_{n+1}^w\|_1+\|K_n^w\|_1}{n+1}.
\]
The Walsh--Paley Lebesgue constants satisfy $\|D_n^w\|_1=O(\log n)$, and
\eqref{eq:Toledo-global} gives the uniform bound for $\|K_n^w\|_1$.  Hence
$\|K_{n+1}^w-K_n^w\|_1\to0$ as $n\to\infty$.  Applying this with $n=a_{m-1}$ gives, as $m\to\infty$,
\[
	\|K_{k_m}^w\|_1\to\frac{17}{15}.
\]
Finally, $k_m/2^m\to2/3$ as $m\to\infty$, and therefore, as $m\to\infty$,
\[
	\frac{k_m\|K_{k_m}^w\|_1}{2^m}\to\frac23\cdot\frac{17}{15}=\frac{34}{45}.
\]
\end{proof}

The next estimate gives the limiting contribution of the principal atoms.

\begin{lemma}\label{L:probe-dominance-alt}
Fix $ v\in\mathbb N$.  Then, for all sufficiently large $m$,
\[
	k_mK_{k_m}^w(e_{m-1- v})\ge Q_{m, v}.
\]
\end{lemma}

\begin{proof}
Put $h:=m-1- v$.  For $y=e_h$,
\[
	kK_k^w(e_h)=\sum_{q=1}^{k}D_q^w(e_h).
\]
The values $D_q^w(e_h)$ are periodic in $q$ with period $2^{h+1}$.  If
$q=2^{h+1}a+s$, $0\le s<2^{h+1}$, then
\[
	D_q^w(e_h)=
	\begin{cases}
	s, & 0\le s\le 2^h,\\
	2^{h+1}-s, & 2^h<s<2^{h+1}.
	\end{cases}
\]
Consequently, if $k=2^{h+1}p+\rho$, $0\le\rho<2^{h+1}$, then
\begin{equation}\label{eq:H-period-formula}
	kK_k^w(e_h)=p2^{2h}+A_h(\rho),
\end{equation}
where
\[
A_h(\rho)=
\begin{cases}
\dfrac{\rho(\rho+1)}2, &0\le\rho\le 2^h,\\[3mm]
\dfrac{2^h(2^h+1)}2+(\rho-2^h)2^h-
\dfrac{(\rho-2^h)(\rho-2^h+1)}2, &2^h<\rho<2^{h+1}.
\end{cases}
\]
Because
\[
	\frac{2^{m+1}}3=\frac{2^{ v+1}}3\,2^{h+1},
\]
write
\[
	\frac{2^{ v+1}}3=p_ v+\theta_ v,
	\qquad
	p_ v:=\left\lfloor\frac{2^{ v+1}}3\right\rfloor,
	\qquad
	\theta_ v\in\left\{\frac13,\frac23\right\}.
\]
Then
\[
	k_m=2^{h+1}p_ v+\rho_m,
	\qquad
	\rho_m=\lceil2^{h+1}\theta_ v\rceil.
\]
If $ v$ is even, then $\theta_ v=2/3$ and
$\rho_m/2^h\to4/3$ as $m\to\infty$.  Using
\eqref{eq:H-period-formula}, we obtain, as $m\to\infty$,
\[
	\frac{k_mK_{k_m}^w(e_h)}{2^{2h}}
	\to
	\frac{3\cdot2^{ v+1}+1}{9}.
\]
If $ v$ is odd, then $\theta_ v=1/3$ and
$\rho_m/2^h\to2/3$ as $m\to\infty$.  Again, as $m\to\infty$,
\[
	\frac{k_mK_{k_m}^w(e_h)}{2^{2h}}
	\to
	\frac{3\cdot2^{ v+1}-1}{9}.
\]
Thus in both cases, as $m\to\infty$,
\[
	\frac{k_mK_{k_m}^w(e_{m-1- v})}{2^{2h}}
	\to
	c_ v:=\frac{3\cdot2^{ v+1}+(-1)^ v}{9}.
\]
On the other hand, since $2^{2h}=2^{2m-2-2 v}$, we have, as $m\to\infty$,
\[
	\frac{Q_{m, v}}{2^{2h}}\to\frac{2^ v}{3}.
\]
Finally,
\[
	c_ v-\frac{2^ v}{3}=\frac{3\cdot2^ v+(-1)^ v}{9}>0.
\]
Therefore $k_mK_{k_m}^w(e_{m-1- v})\ge Q_{m, v}$ for all sufficiently large
$m$.
\end{proof}

We compute the limiting principal gain along the alternating sequence.

\begin{lemma}\label{L:principal-gain-alt}
For $k_m$ given by \eqref{eq:def-km}, we have, as $m\to\infty$,
\[
	\frac{\mathcal A_m(k_m)}{2^m}\to\frac7{18}.
\]
\end{lemma}

\begin{proof}
By definition,
\[
	\frac{\mathcal A_m(k_m)}{2^m}
	=\frac{k_m(k_m+1)}{2\cdot2^{2m}}+
	\frac1{2^{2m}}\sum_{ v=0}^{m-1}
	\min\left(Q_{m, v},|k_mK_{k_m}^w(e_{m-1- v})|\right).
\]
The first term tends to $2/9$.  Denote the second term by $U_m$.
For any fixed $N$, Lemma~\ref{L:probe-dominance-alt} gives, for all sufficiently
large $m$,
\[
	\min\left(Q_{m, v},|k_mK_{k_m}^w(e_{m-1- v})|\right)=Q_{m, v},
	\qquad 0\le v\le N.
\]
Since $Q_{m, v}/2^{2m}\to1/(12\cdot2^ v)$ as $m\to\infty$,
\[
	\liminf_{m\to\infty}U_m
	\ge\sum_{ v=0}^{N}\frac1{12\cdot2^ v}.
\]
Letting $N\to\infty$ gives $\liminf U_m\ge1/6$.

For the reverse inequality, split the sum into $0\le v\le N$ and $ v>N$.
The first part is bounded by the corresponding sum of $Q_{m, v}/2^{2m}$.  For
the tail, use $\min(a,b)\le b$ and
\[
	0\le k_mK_{k_m}^w(e_h)\le k_m2^h\le2^m2^h.
\]
With $h=m-1- v$, this gives
\[
	k_mK_{k_m}^w(e_{m-1- v})\le\frac{2^{2m}}{2^{ v+1}}.
\]
Hence the tail is at most $2^{-N-1}$.  Letting first $m\to\infty$ and then
$N\to\infty$, we get $\limsup U_m\le1/6$.  Therefore $U_m\to1/6$ as $m\to\infty$, and
\[
	\frac{\mathcal A_m(k_m)}{2^m}\to\frac29+\frac16=\frac7{18}
	\qquad(m\to\infty).
\]
\end{proof}

The next elementary estimate controls Walsh--Paley Dirichlet kernels at points with one prescribed non-zero coordinate.

\begin{lemma}\label{L:Dirichlet-one-bit-bound} Let $y\in G$, and suppose that $y_r=1$ for some $r\in\mathbb N$. Then 
\begin{equation}\label{eq} |D_n^w(y)|\le2^r\qquad(n\in\mathbb P).
\end{equation}
\end{lemma} 
\begin{proof} Let $b\in\mathbb N$ and $0\le a<2^r$. Then \[ w_{b2^{r+1}+a+2^r}(y) = w_{b2^{r+1}+a}(y)w_{2^r}(y) = -w_{b2^{r+1}+a}(y), \] because $w_{2^r}=r_r$ and $y_r=1$. Hence the terms cancel in pairs in every block of length $2^{r+1}$. Thus every complete block has sum zero. The remaining incomplete block has at most $2^r$ uncancelled terms. Since $|w_j(y)|=1$ for all $j$, we get \eqref{eq}.
\end{proof}

The remaining atoms have negligible contribution on the scale needed below.

\begin{lemma}\label{L:nonprincipal-negligible}
Let $k\in\mathbb P$ and suppose that $k<2^m$.  Then
\[
	0\le\mathcal G_m(k)-\mathcal A_m(k)
	\le Cm2^{2m/3}
\]
with an absolute constant $C>0$.  In particular, since $k_m<2^m$ for $m\ge2$, we have, as $m\to\infty$,
\[
	\frac{\mathcal G_m(k_m)-\mathcal A_m(k_m)}{2^m}\to0.
\]
\end{lemma}

\begin{proof}
The quantity $\mathcal A_m(k)$ is exactly the contribution of $I_m$ and of
$I_m(e_ v)$, $0\le v\le m-1$, to $\mathcal G_m(k)$.  Hence
$\mathcal G_m(k)-\mathcal A_m(k)$ is the contribution of the points whose first
$m$ coordinates contain at least two $1$'s.

Partition this set according to the first two non-zero coordinates.  For each
pair $0\le s<j\le m-1$, the set $I_{j+1}(e_s+e_j)$ consists exactly of those
$x$ for which
\[
	x_s=x_j=1,
	\qquad
	x_i=0\quad(0\le i\le j,
\ i\ne s,j).
\]
These sets form a pairwise disjoint partition of the set in question, and
$\mu(I_{j+1}(e_s+e_j))=2^{-j-1}$.

On $I_{j+1}(e_s+e_j)$ one has $D_{2^m}=0$.  By \eqref{eq:F-on-Esj},
\[
	|2^mK_{2^m}^{\kappa}(x)|
	\le C(2^{2s}+2^{2j-s}),
	\qquad x\in I_{j+1}(e_s+e_j).
\]
If $x\in I_{j+1}(e_s+e_j)$, then $(\tau_m(x))_{m-1-j}=1$.  By
Lemma~\ref{L:Dirichlet-one-bit-bound}, for every $r\in\mathbb P$,
\[
	|D_r^w(\tau_m(x))|\le2^{m-1-j}.
\]
Thus
\[
	|kK_k^w(\tau_m(x))|
	=\left|\sum_{r=1}^{k}D_r^w(\tau_m(x))\right|
	\le k2^{m-1-j}
	\le2^{2m-1-j}.
\]
It follows that
\[
\begin{aligned}
	\frac{\mathcal G_m(k)-\mathcal A_m(k)}{2^m}
	&\le
	C\sum_{0\le s<j\le m-1}
	\left[
	\min(2^{2s-m-j},2^{m-2j})
	+
	\min(2^{j-m-s},2^{m-2j})
	\right].
\end{aligned}
\]
Both double sums are $O(m2^{-m/3})$.  Indeed, if $j\le2m/3$, use the first
term inside each minimum and sum the resulting geometric series.  If
$j>2m/3$, use the second term.  This gives a contribution bounded by
$Cm2^{-m/3}$.  Therefore
\[
	\frac{\mathcal G_m(k)-\mathcal A_m(k)}{2^m}\le Cm2^{-m/3},
\]
which proves the estimate.
\end{proof}

We can now prove the lower sharpness estimate.

\begin{theorem}\label{T:lower-71-50}
For $k_m$ given by \eqref{eq:def-km},
\[
	\lim_{m\to\infty}\|K_{2^m+k_m}^{\kappa}\|_1=\frac{71}{50}.
\]
Consequently,
\[
	\sup_{n\in\mathbb P}\|K_n^{\kappa}\|_1\ge\frac{71}{50}.
\]
\end{theorem}

\begin{proof}
For all $m\ge2$ we have $1\le k_m<2^m$, so the gain formula is applicable.  By \eqref{eq:gain-formula},
\[
	\|K_{2^m+k_m}^{\kappa}\|_1
	=
	\frac{2^m\|K_{2^m}^{\kappa}\|_1+k_m+k_m\|K_{k_m}^w\|_1-\mathcal G_m(k_m)}{2^m+k_m}.
\]
Using Corollary~\ref{third} and the preceding lemmas, as $m\to\infty$,
\[
	\|K_{2^m}^{\kappa}\|_1\to\frac43,
	\qquad
	\frac{k_m}{2^m}\to\frac23,
\]
\[
	\frac{k_m\|K_{k_m}^w\|_1}{2^m}\to\frac{34}{45},
	\qquad
	\frac{\mathcal G_m(k_m)}{2^m}\to\frac7{18}.
\]
Therefore
\[
	\lim_{m\to\infty}\|K_{2^m+k_m}^{\kappa}\|_1
	=
	\frac{\frac43+\frac23+\frac{34}{45}-\frac7{18}}{1+\frac23}
	=\frac{71}{50}.
\]
\end{proof}

\subsection{Principal Paley domination}

The first auxiliary estimate concerns the case $m=|k|+1$.

\begin{lemma}\label{L:one-step-P-lower}
For every $k\in\mathbb P$,
\[
	\mathcal A_{|k|+1}(k)
	\ge
	\frac{k^2}{2^{|k|+2}}+\frac{2^{|k|}}3.
\]
\end{lemma}

\begin{proof}
The atom $I_{|k|+1}$ contributes
\[
	\frac{k(k+1)}{2\cdot2^{|k|+1}}\ge\frac{k^2}{2^{|k|+2}}.
\]
It remains to show that the contribution of the singleton atoms
$I_{|k|+1}(e_ v)$, $0\le  v\le |k|$, is at least $2^{|k|}/3$.

For $0\le h\le |k|$, the quantities $D_u^w(e_h)$ are non-negative for
$1\le u\le k$.  Hence $kK_k^w(e_h)=\sum_{u=1}^kD_u^w(e_h)$ is increasing in
$k$.  Therefore the singleton contribution is bounded from below by its value
at the left endpoint $k=2^{|k|}$ of the dyadic block.

Put $h:=|k|- v$.  For $k=2^{|k|}$ we have
\[
	2^{|k|}K_{2^{|k|}}^w(e_h)
	=
	\begin{cases}
		2^{|k|}2^{h-1}, & 0\le h\le |k|-1,\\[1mm]
		\dfrac{2^{|k|}(2^{|k|}+1)}2, & h=|k|.
	\end{cases}
\]
Moreover,
\[
	Q_{|k|+1,|k|-h}
	=2^{2|k|-2h}+\frac{2^{|k|+h}-2^{|k|-h}}3.
\]
For $0\le h\le |k|-1$, both
$Q_{|k|+1,|k|-h}$ and $2^{|k|}2^{h-1}$ are not smaller than
\[
	\frac{2^{|k|+h}}3+\frac{2^{|k|-h}}6.
\]
Indeed, the first inequality follows from
\[
	2^{2|k|-2h}-\frac{2^{|k|-h}}3\ge\frac{2^{|k|-h}}6,
\]
and the second one follows from
\[
	2^{|k|}2^{h-1}-\frac{2^{|k|+h}}3
	=\frac{2^{|k|+h}}6
	\ge \frac{2^{|k|-h}}6.
\]
For $h=|k|$ we use
\[
	Q_{|k|+1,0}=\frac{2^{2|k|}}3+\frac23
	\le \frac{2^{|k|}(2^{|k|}+1)}2,
\]
and hence
\[
	\min\left(Q_{|k|+1,0},\frac{2^{|k|}(2^{|k|}+1)}2\right)
	=\frac{2^{2|k|}}3+\frac23.
\]
Consequently, the sum of the singleton terms is at least
\[
\begin{aligned}
&\sum_{h=0}^{|k|-1}
\left(\frac{2^{|k|+h}}3+\frac{2^{|k|-h}}6\right)
+\frac{2^{2|k|}}3+\frac23 \\
&\qquad=
\frac{2^{2|k|+1}}3+\frac13
\ge \frac{2^{2|k|+1}}3.
\end{aligned}
\]
After division by $2^{|k|+1}$, the singleton contribution to
$\mathcal A_{|k|+1}(k)$ is at least $2^{|k|}/3$.  This, together with the
contribution of $I_{|k|+1}$, proves the lemma.
\end{proof}

We first prove the domination estimate in the one-step case.

\begin{proposition}\label{P:one-step-reserve}
For every $k\in\mathbb P$,
\[
	k\|K_k^w\|_1
	\le
	\mathcal A_{|k|+1}(k)
	+\frac{13}{150}2^{|k|+1}
	+\frac{21}{50}k.
\]
\end{proposition}

\begin{proof}
Put $t:=k/2^{|k|}$.  Then $1\le t<2$.  By
Lemma~\ref{L:one-step-P-lower},
\[
	\frac{\mathcal A_{|k|+1}(k)}{2^{|k|}}\ge\frac{t^2}{4}+\frac13.
\]
If $0\le k-2^{|k|}\le2^{|k|}/3$, then $1\le t\le4/3$.  Toledo's global bound gives
$k\|K_k^w\|_1\le(17/15)k$, and hence
\[
\begin{aligned}
&\frac1{2^{|k|}}\left(
\mathcal A_{|k|+1}(k)+\frac{13}{150}2^{|k|+1}
+\frac{21}{50}k-k\|K_k^w\|_1
\right) \\
&\qquad\ge
\frac{t^2}{4}+\frac13+\frac{13}{75}+\frac{21}{50}t-\frac{17}{15}t
=\frac{(3t-4)(25t-38)}{300}\ge0.
\end{aligned}
\]
If $2^{|k|}/3\le k-2^{|k|}<2^{|k|}$, then $4/3\le t<2$.  By
Lemma~\ref{L:blockwise-Toledo-excess},
$k\|K_k^w\|_1\le k+(8/45)2^{|k|}$.  Therefore
\[
\begin{aligned}
&\frac1{2^{|k|}}\left(
\mathcal A_{|k|+1}(k)+\frac{13}{150}2^{|k|+1}
+\frac{21}{50}k-k\|K_k^w\|_1
\right) \\
&\qquad\ge
\frac{t^2}{4}+\frac13+\frac{13}{75}+\frac{21}{50}t-t-\frac8{45}
=\frac{(3t-4)(75t-74)}{900}\ge0.
\end{aligned}
\]
\end{proof}

The next proposition treats the two-step case.

\begin{proposition}\label{P:two-step-close-gap}
For every $k\in\mathbb P$,
\[
	k\|K_k^w\|_1
	\le
	\mathcal A_{|k|+2}(k)
	+\frac{13}{150}2^{|k|+2}
	+\frac{21}{50}k.
\]
\end{proposition}

\begin{proof}
Put $t:=k/2^{|k|}$.  Then $1\le t<2$.

The atom $I_{|k|+2}$ contributes at least $2^{|k|}t^2/8$.  The singleton with
$h=|k|+1$ has
\[
	kK_k^w(e_{|k|+1})=\frac{k(k+1)}2,
	\qquad
	Q_{|k|+2,0}=\frac{4\cdot2^{2|k|}+2}{3},
\]
and contributes at least
\[
	\frac{2^{|k|}}4\min\left(\frac43,\frac{t^2}{2}\right).
\]
For $h=|k|$,
\[
\begin{aligned}
	kK_k^w(e_{|k|})
	&=\frac{2^{|k|}(2^{|k|}+1)}2
	 +(k-2^{|k|})2^{|k|}
	 -\frac{(k-2^{|k|})(k-2^{|k|}+1)}2 \\
	&\ge\frac{2^{2|k|}}2,
\end{aligned}
\]
and
\[
	Q_{|k|+2,1}=4+\frac{2\cdot2^{2|k|}-2}{3}\ge\frac{2^{2|k|}}2.
\]
This gives a contribution of at least $2^{|k|}/8$.  Finally, for $0\le h\le |k|-1$,
\[
	kK_k^w(e_h)\ge2^{|k|}2^{h-1},
\]
and the corresponding $Q$-values dominate these quantities.  Therefore the lower
singleton atoms contribute at least
\[
	\frac1{2^{|k|+2}}\sum_{h=0}^{|k|-1}2^{|k|}2^{h-1}
	=\frac{2^{|k|}}8-\frac18.
\]
Combining these bounds gives
\begin{equation}\label{eq:Pp2-lower-simple}
	\mathcal A_{|k|+2}(k)
	\ge
	2^{|k|}\left[
	\frac{t^2}{8}
	+\frac14\min\left(\frac43,\frac{t^2}{2}\right)
	+\frac14
	\right]-\frac18.
\end{equation}

If $|k|=0$, the claim follows directly from
$\mathcal A_2(1)=3/4$ and $\|K_1^w\|_1=1$.  Assume that $|k|\ge1$.  By
Lemma~\ref{L:blockwise-Toledo-excess},
\[
	k\|K_k^w\|_1\le k+\frac8{45}2^{|k|}.
\]
Using \eqref{eq:Pp2-lower-simple}, we obtain
\[
\begin{aligned}
&\frac1{2^{|k|}}\left(
\mathcal A_{|k|+2}(k)+\frac{13}{150}2^{|k|+2}
+\frac{21}{50}k-k\|K_k^w\|_1
\right) \\
&\qquad\ge
\frac{t^2}{8}
+\frac14\min\left(\frac43,\frac{t^2}{2}\right)
+\frac14
-\frac1{8\cdot2^{|k|}}
-\frac{29}{50}t
+\frac{38}{225}.
\end{aligned}
\]
If $1\le t\le\sqrt{8/3}$, the last expression without the negative term
$1/(8\cdot2^{|k|})$ equals
\[
	\frac{t^2}{4}-\frac{29}{50}t+\frac14+\frac{38}{225},
\]
whose minimum is $464/5625$.  If $\sqrt{8/3}\le t<2$, the corresponding
expression is
\[
	\frac{t^2}{8}-\frac{29}{50}t+\frac7{12}+\frac{38}{225},
\]
which is decreasing on this interval and has value $83/900$ at $t=2$.  Since
\[
	\min\left(\frac{464}{5625},\frac{83}{900}\right)>\frac1{16}
	\ge\frac1{8\cdot2^{|k|}},
\]
the desired inequality follows.
\end{proof}

For larger gaps we use the following lower bound for $\mathcal A_m$.

\begin{lemma}\label{L:Pm-lower-away}
Let $k\in\mathbb P$ and suppose that $m-|k|\ge3$.  Then
\[
	\mathcal A_m(k)
	\ge
	\frac{(m-|k|)k(k+1)}{2\cdot2^m}
	+\frac{2^{2|k|}}{2^m}.
\]
\end{lemma}

\begin{proof}
The atom $I_m$ contributes $k(k+1)/(2\cdot2^m)$.  Write $h=m-1- v$.  If
$h\ge |k|+1$, equivalently $ v\le m-|k|-2$, then $k<2^{|k|+1}\le2^h$, so
\[
	kK_k^w(e_h)=\frac{k(k+1)}2.
\]
The corresponding $Q$-values dominate this number.  Thus these $m-|k|-1$
singleton atoms, together with $I_m$, give
\[
	\frac{(m-|k|)k(k+1)}{2\cdot2^m}.
\]
For $0\le h\le |k|-1$,
\[
	kK_k^w(e_h)\ge2^{|k|}K_{2^{|k|}}^w(e_h)=2^{|k|}2^{h-1},
\]
and the corresponding $Q$-values dominate these quantities.  Their total
contribution is at least
\[
	\frac1{2^m}\sum_{h=0}^{|k|-1}2^{|k|}2^{h-1}
	=\frac{2^{2|k|}-2^{|k|}}{2\cdot2^m}.
\]
For $h=|k|$,
\[
	kK_k^w(e_{|k|})
	\ge2^{|k|}K_{2^{|k|}}^w(e_{|k|})
	=\frac{2^{|k|}(2^{|k|}+1)}2,
\]
and the corresponding atom contributes at least
$2^{|k|}(2^{|k|}+1)/(2\cdot2^m)$.  The last two contributions add up to
$2^{2|k|}/2^m$.
\end{proof}

The following proposition covers the case where the gap is at least three.

\begin{proposition}\label{P:gap-at-least-three}
Let $k\in\mathbb P$ and suppose that $m-|k|\ge3$.  Then
\[
	k\|K_k^w\|_1
	\le
	\mathcal A_m(k)+\frac{13}{150}2^m+\frac{21}{50}k.
\]
\end{proposition}

\begin{proof}
Put $d:=m-|k|$ and $t:=k/2^{|k|}$.  Then $d\ge3$ and $1\le t<2$.  By
Lemma~\ref{L:blockwise-Toledo-excess},
\[
	k\|K_k^w\|_1\le k+\frac8{45}2^{|k|}.
\]
By Lemma~\ref{L:Pm-lower-away},
\[
	\mathcal A_m(k)
	\ge\frac{dk(k+1)}{2\cdot2^m}+\frac{2^{2|k|}}{2^m}.
\]
Therefore
\[
\begin{aligned}
&\frac1{2^{|k|}}\left(
\mathcal A_m(k)+\frac{13}{150}2^m+\frac{21}{50}k-k\|K_k^w\|_1
\right) \\
&\qquad\ge
\frac{dt^2}{2^{d+1}}+\frac1{2^d}
+\frac{13}{150}2^d-\frac{29}{50}t-\frac8{45}.
\end{aligned}
\]
For $d=3$, the right-hand side is
\[
	\frac{3t^2}{16}+\frac18+\frac{52}{75}-\frac{29}{50}t-\frac8{45},
\]
whose minimum on $[1,2]$ is $8641/45000>0$.  If $d\ge4$, the expression is
decreasing in $t\in[1,2)$, and its value at $t=2$ is
\[
	\frac{2d+1}{2^d}+\frac{13}{150}2^d-\frac{29}{25}-\frac8{45}.
\]
This equals $2201/3600>0$ at $d=4$ and is increasing for $d\ge4$.  Hence the
inequality holds for all $d\ge3$.
\end{proof}

The three previous propositions imply the principal Paley domination estimate.

\begin{theorem}\label{T:principal-Paley-domination-final}
For every $m\in\mathbb N$ and every $k\in\mathbb P$ with $k<2^m$,
\[
	k\|K_k^w\|_1
	\le
	\mathcal A_m(k)+\frac{13}{150}2^m+\frac{21}{50}k.
\]
\end{theorem}

\begin{proof}
Since $k<2^m$, we have $|k|\le m-1$.  If $m=|k|+1$, use
Proposition~\ref{P:one-step-reserve}.  If $m=|k|+2$, use
Proposition~\ref{P:two-step-close-gap}.  If $m-|k|\ge3$, use
Proposition~\ref{P:gap-at-least-three}.  These cases cover all possibilities.
\end{proof}

\subsection{The exact global supremum}

We now prove the sharp global upper estimate.

\begin{theorem}\label{T:sharp-upper-71-50}
For every $n\in\mathbb P$,
\[
	\|K_n^{\kappa}\|_1\le\frac{71}{50}.
\]
\end{theorem}

\begin{proof}
If $n=2^{|n|}$, then Theorem~\ref{T:exact_upper_limit} and Corollary~\ref{third}
give
\[
	\|K_n^{\kappa}\|_1=\|K_{2^{|n|}}^{\kappa}\|_1\le\frac43<\frac{71}{50}.
\]
Assume now that $n$ is not a power of two, and put
\[
	k:=n-2^{|n|}.
\]
Then $1\le k<2^{|n|}$.  By the gain formula,
\[
	n\|K_n^{\kappa}\|_1
	=
	2^{|n|}\|K_{2^{|n|}}^{\kappa}\|_1
	+k+k\|K_k^w\|_1-
	\mathcal G_{|n|}(k).
\]
Using $\|K_{2^{|n|}}^{\kappa}\|_1\le4/3$, Lemma~\ref{L:principal-gain}, and
Theorem~\ref{T:principal-Paley-domination-final}, we obtain
\[
\begin{aligned}
	n\|K_n^{\kappa}\|_1
	&\le
	\frac43 2^{|n|}+k+k\|K_k^w\|_1-
	\mathcal A_{|n|}(k) \\
	&\le
	\frac43 2^{|n|}+k+\frac{13}{150}2^{|n|}+\frac{21}{50}k \\
	&=\frac{71}{50}(2^{|n|}+k)=\frac{71}{50}n.
\end{aligned}
\]
Therefore $\|K_n^{\kappa}\|_1\le71/50$.
\end{proof}

The exact value of the global supremum follows immediately.

\begin{theorem}\label{T:exact-supremum-71-50}
We have
\[
	\sup_{n\in\mathbb P}\|K_n^{\kappa}\|_1=\frac{71}{50}.
\]
\end{theorem}

\begin{proof}
The upper bound follows from Theorem~\ref{T:sharp-upper-71-50}.  The opposite
inequality follows from Theorem~\ref{T:lower-71-50}.  Hence the supremum is
$71/50$.
\end{proof}

\end{document}